\documentclass[reqno]{amsart}

\usepackage{a4,amsmath,amssymb,amscd,verbatim,bbm,graphicx}
    \usepackage{color}
\newtheorem{theorem}{Theorem}
\newtheorem{proposition}[theorem]{Proposition}
\newtheorem{corollary}[theorem]{Corollary}
\newtheorem{lemma}[theorem]{Lemma}

\theoremstyle{definition}

\newtheorem{definition}[theorem]{Definition}

\newtheorem{remark}[theorem]{Remark}

\numberwithin{equation}{section}
\numberwithin{theorem}{section}

\title{Relative Growth in Hyperbolic Groups}

\begin{document}
\bibliographystyle{plain}

\author{Stephen Cantrell}
\address{Mathematics Institute, University of Warwick,
Coventry CV4 7AL, U.K.}
\email{S.J.Cantrell@warwick.ac.uk}

\author{Richard Sharp} 
\address{Mathematics Institute, University of Warwick,
Coventry CV4 7AL, U.K.}
\email{R.J.Sharp@warwick.ac.uk}
\maketitle

\begin{abstract}
In this note we obtain estimates on the relative growth of normal subgroups of 
non-elementary hyperbolic groups, particularly those with 
free abelian quotient.
As a corollary, we deduce 
that the associated relative growth series fail to be rational.
\end{abstract}

\section{Introduction and Results}
Let $G$ be a non-elementary hyperbolic group equipped with a finite symmetric generating set. Write $W_n = \{ g \in G: |g| =n \}$ for the collection of elements of word length $n$. By a result of Coornaert \cite{coor}, the growth rate of its cardinality $\# W_n$ is purely exponential,
i.e. there exist constants  $\lambda >1$ and $C_1,C_2 >0$ such that
$$ C_1 \lambda^n \le \#W_n \le C_2 \lambda^n$$
for all $n \ge 1$. Now suppose that $N$ is a subgroup of $G$. An interesting question to ask is how 
$\#(W_n\cap N)$, 
which we call the relative growth of $N$,
grows in comparison to $\#W_n$. A result of Gou\"ezel, Matheus and Maucourant \cite{gmm}
states that if $N$ has infinite index in $G$ then 
\begin{equation} \label{gmm}
\lim_{n \to \infty} \frac{\#(W_n \cap N)}{\#W_n} = 0.
\end{equation}
This is a subtle result that relies strongly on the hyperbolicity of $G$. If we suppose further that $N$ is normal and the quotient $G/N$ is isomorphic to $\mathbb Z^\nu$ for some $\nu \ge 1$, then we have access to more structure. With this additional information it seems reasonable to expect that we can describe the relative
growth of $N$ more precisely. 

Pollicott and Sharp \cite{ps1996} studied this problem when $G$ is 
the fundamental groups
of a compact orientable surface of genus at least two and $N$ is the commutator subgroup. 
Sharp \cite{sharp} extended this to cover hyperbolic groups $G$ that may be
realised as 
convex cocompact groups of isometries of real hyperbolic space whose fundamental domain can be chosen to be a finite sided polyhedron $R$
such that
$\bigcup_{g \in G} \partial R$ is a union of geodesic hyperplanes, with generators given by the side pairings.
The fundamental groups of compact surfaces
were shown to satisfy this condition by Bowen and Series \cite{bs}.
In addition, this class includes free groups on at least two generators and certain higher dimensional examples (see Bourdon's thesis \cite{bourdon}).
In these cases, it was shown that there exists an integer $D \ge 1$ such that, along the subsequence $Dn$, the relative growth $\#(W_{Dn} \cap N)$ grows asymptotically like $\lambda^{Dn}/(Dn)^{\nu/2}$, as $n \to \infty$. The aim of this note is to extend this result so that it applies all non-elementary hyperbolic groups. 

Before we state our main result, we recall the following standard definitions.
Given two real valued sequences $a_n$ and $b_n$, we say that $a_n \sim b_n$ if $a_n/b_n \to 1$, as $n \to\infty$. Furthermore, if $b_n$ is positive, we say that 
$a_n = O(b_n)$ if there exists a constant $C>0$ such that $|a_n| \le Cb_n$, for all $n \ge 1$.

\begin{theorem} \label{thm}
Let $G$ be a non-elementary hyperbolic group equipped with a finite symmetric generating set and let $N \triangleleft G$ be a normal subgroup with $G/N \cong \mathbb Z^\nu$ for some $\nu \ge 1$. Then
$$\# ( W_n \cap N) = O\left(\frac{\lambda^n}{n^{\nu/2}}\right)$$
as $n\to\infty$. Furthermore, there exists $D \in \mathbb{Z}_{\ge 0}$ and $C>0$ such that
$$\#(W_{Dn} \cap N) \sim \frac{C \lambda^{Dn}}{(Dn)^{\nu/2}}$$ 
as $n\to\infty$.
\end{theorem}

This theorem has the following immediate corollary.

\begin{corollary} \label{cor1}
Let $G$ be a non-elementary hyperbolic group equipped with a finite symmetric generating set and let $N \triangleleft G$ be a normal subgroup
such that the abelianisation of $G/N$ has rank $\nu \ge 1$.
Then
$$\# ( W_n \cap N) = O\left(\frac{\lambda^n}{n^{\nu/2}}\right)$$
as $n\to\infty$.
\end{corollary}

\begin{proof}
Write the abelianisation of $G/N$ as $\mathbb Z^\nu \times F$, where $F$ is finite. 
There are then natural surjective homomorphisms $\phi : G \to G/N$ and $\psi : G/N \to \mathbb Z^\nu$.
Set $\phi_0 = \psi \circ \phi$ and $N_0 = \ker \phi_0$. 
Then $N \subset N_0$. Furthermore, by Theorem \ref{thm},
$\#(W_n \cap N_0) = O(\lambda^n n^{-\nu/2})$, giving the required estimate.
\end{proof}

\begin{remark}
The relative growth in Corollary \ref{cor1} may occur at a slower exponential rate.
Indeed, Coulon, Dal'Bo and Sambusetti recently showed that $\#(W_n \cap N) = O(\lambda_0^n)$, for some $0< \lambda_0 <\lambda$ precisely when
$G/N$ is {\it not} amenable \cite{cds}. In fact, their result does not require normality of the subgroup, in 
which case amenability is replaced by co-amenability of $N$ in $G$, i.e. that the $G$-action on the coset space
$G/N$ is amenable.
\end{remark}
To prove Theorem \ref{thm}, we would like to employ the strategy used by the second author in \cite{sharp}. However, there are significant technical obstacles which we need to overcome in order to use this method. We summarise these below.\\ 
(i) Firstly, as mentioned above, in \cite{sharp} there are strong restrictions on the hyperbolic groups and their generating sets. This makes it much easier to study the relative growth quantity $\#(W_n\cap N)$. In the current paper we need to find a new approach that works for general non-elementary hyperbolic groups, that will allow us to express $\#(W_n\cap N)$ in terms of quantities which we can analyse. To achieve this we appeal to ideas and techniques used in \cite{stats}.\\ 
(ii) Secondly, we need a good understanding of how real valued group homomorphisms on hyperbolic groups grow as we increase the word length of the input. Again, recent work of the first author \cite{stats} allows us to deduce the required properties of these homomorphisms.\\

We end this section with a discussion of relative growth series. We define the relative growth series for $N$ in
$G$ (with respect to the given generators) to be the power series
\[
 \sum_{n=0}^\infty \#(W_n \cap N) z^n.
\]
When $N =G$, this is the standard growth series and, for hyperbolic groups, is well-known to be 
the series of a rational function \cite{can}, \cite{gh}.
The requirement that a power series be rational imposes a strong constraint on the coefficients:
if $\sum_{n=0}^\infty a_n z^n$ is rational then
there are complex numbers $\xi_1,\ldots,\xi_m$ and polynomials $P_1,\ldots,P_m$ such that
\[
a_n = \sum_{j=1}^m P_j(n) \xi_j^n
\]
(Theorem IV.9 of \cite{fs}). Comparing with the asymptotic in Theorem \ref{thm}, we see that 
$\#(W_n \cap N)$ does not satisfy this constraint. Thus we obtain the following.

\begin{corollary} \label{cor2}
Suppose $G$ is a non-elementary hyperbolic group equipped with a finite symmetric generating set. Let $N \triangleleft G$ be a normal subgroup with $G/N \cong \mathbb{Z}^\nu$, 
for some $\nu \ge 1$. Then, the relative
growth series
$$\sum_{n=1}^\infty \#(W_n \cap N) z^n$$
is not the series of a rational function.
\end{corollary}

\begin{remark} (i) The first result of this type is due to Grigorchuk, who showed that the relative growth series 
is not rational when $G$ is the free group on two generators and $N$ is the commutator subgroup
(see \cite{gri-harpe}). A similar result was obtained for the fundamental groups of compact surfaces of genus
$\ge 2$ in \cite{ps1996} and this was extended to a wider class of hyperbolic groups in \cite{sharp}.

\noindent
(ii)
We note that, as Corollary \ref{cor2} requires the asymptototic along a subsequence in Theorem \ref{thm},
it does not apply to general infinite index subgroups of hyperbolic groups.
In fact, Grigorchuk showed that  if $N$ is a finite index subgroup of a free group than its relative growth series 
is rational \cite{gri}.
\end{remark}

\section{Preliminaries}
We first recall the definition of a hyperbolic group.
A metric space is hyperbolic if there exist $\delta \ge 0$ for which every geodesic triangle is $\delta$-thin, i.e. given any geodesic triangle, the union of the $\delta$ neighbourhoods of any two sides of this triangle contain the third side. A finitely generated group $G$ is said to be hyperbolic, if given any finite generating set $S$, the Cayley graph of $G$ with respect to $S$ is a hyperbolic metric space when equipped with the word metric.
We say that a hyperbolic group is elementary if it contains a cyclic subgroup of finite index. We will be 
exclusively concerned with non-elementary hyperbolic groups.
 \\ \indent
Hyperbolic groups have nice combinatorial properties that arise due to their strongly Markov structure.

\begin{definition} \label{lab}
A finitely generated group $G$ is strongly Markov if given any  generating set $S$ there exists a finite directed graph $\mathcal{G}$ with vertex set $V$, edge set $E$ 
(with at most one directed edge between an ordered pair of vertices) and a labeling map $\rho : E \to S$ such that:
\begin{enumerate}
\item there exists an initial vertex $\ast \in V$ such that no directed edge ends at $\ast$;
\item the map taking finite paths in $\mathcal{G}$ starting at $\ast$ to $G$ that sends a path with concurrent edges $(\ast,x_1), \ldots ,(x_{n-1},x_n)$ to $\rho(\ast,x_1)\rho(x_1,x_2) \cdots \rho(x_{n-1},x_n)$, is a bijection; 
\item the word length of $\rho(\ast,x_1) \cdots \rho(x_{n-1},x_n)$ is $n$.
\end{enumerate}
\end{definition}

In \cite{gh} Ghys and de le Harpe extended Cannon's work on Kleinian groups \cite{can} and proved that hyperbolic groups are strongly Markov.

\begin{proposition} [\cite{gh}, Chapitre 9, Th\'eor\`eme 13]
Any hyperbolic group is strongly Markov.
\end{proposition}
Suppose that $\mathcal{G} = (E,V)$ is a directed graph associated to $G$ satisfying the properties in Definition \ref{lab}. We define a transition matrix $A$, indexed by $V \times V$,  by
\[
  A(v_1,v_2) = \left\{
     \begin{array}{@{}l@{\thinspace}l}
        1 & \ \  \text {if} \hspace{2mm}  (v_1,v_2) \in E \\
        0 & \ \  \text{otherwise.}
     \end{array}
   \right.
\]
Using $A$ we define a space
$$
\Sigma_A = \{ (x_n)_{n=0}^\infty : x_n \in V \text{ and } A(x_n,x_{n+1}) = 1
 \text{ for all } n\in\mathbb{Z}_{\ge 0}\}
$$
and $\sigma:\Sigma_A \to \Sigma_A$ by $\sigma((x_n)_{n=0}^\infty) = (x_{n+1})_{n=0}^\infty$. The system $(\Sigma_A,\sigma)$ is known as a subshift of finite type. 

Recall that a matrix $M$ with zero-one entries is called irreducible if for each $i,j$ there exists $n(i,j)$ for which $M^{n(i,j)}(i,j) >0$. 
This is equivalent to the directed graph $\mathcal G$ being connected.
We call $M$ aperiodic if there exists $n$ such that every entry of $M^n$ is strictly positive. Due to the $\ast$ vertex, which forms its own connected component in $\mathcal G$,
$A$ is never irreducible. However, it is possible that, after removing from $A$ the row and column corresponding to the $\ast$ state, the resulting matrix is aperiodic. In fact, for the hyperbolic groups and generating sets considered by Sharp in \cite{sharp}, it is always possible to find  a corresponding directed graph described by an aperiodic matrix (after removing $\ast$). This is not true in general and to improve upon the results in \cite{sharp}, we need to exploit geometrical and combinatorial properties of hyperbolic groups to obtain additional structural information about the directed graph $\mathcal{G}$. Throughout the rest of this section we introduce the preliminaries that will allow us to analyse $\#(W_n\cap N)$ for general hyperbolic groups. \\ \indent

As mentioned above, in general, the graph $\mathcal G$ may have several connected components. By relabeling the vertex set $V$, we may assume that $A$ has the form
$$A = \begin{pmatrix} 
A_{1,1} & 0 & \dots & 0  \\
A_{2,1} & A_{2,2} & \dots & 0\\
\vdots & \vdots & \ddots & \vdots\\
A_{m,1} & A_{m,2} & \dots & A_{m,m}
\end{pmatrix},$$
where each $A_{j,j}$ is irreducible for $j=1,...,m$. We call the $A_{j,j}$ the irreducible components of $A$. 

Let $\lambda>1$ denote the exponential growth rate of $W_n$. It is easy to see by Property $(2)$ and $(3)$ in Definition \ref{lab}, that all of the $A_{j,j}$ must have spectral radius at most $\lambda$. Furthermore there must be at least one $A_{j,j}$ with spectral radius exactly $\lambda$. We call an irreducible component maximal if it has spectral radius $\lambda$. We label the maximal components $B_j$ for $j=1,\ldots ,m$ and define $\Sigma_{B_j}$, $j=1,\ldots ,m$ analogously to $\Sigma_A$. For each $\Sigma_{B_j}$ there exists $p_j \ge 1$ such that $\Sigma_{B_j}$ admits a cyclic decomposition into $p_j$ disjoint sets, 
$$
\Sigma_{B_j} = \bigsqcup_{k=0}^{p_j-1} \Sigma_k^j.
$$
We call $p_j$ the cyclic period of $\Sigma_{B_j}$. The shift map $\sigma$ sends $\Sigma_k^j$ into $\Sigma_{k+1}^j$ where $k,k+1$ are taken modulo $p_j$. Hence each $\Sigma_k^j$ is $\sigma^{p_j}$-invariant. In fact, each system $\sigma^{p_j} : \Sigma_k^j \to \Sigma_k^j$ is a 
subshift of finite type with aperiodic transition matrix. 

The following key result, that relies on Coornaert's estimates for $\#W_n$, shows that the maximal components $B_j$ do not interact with each other. This result allows us to gain a better understanding of the structure of $\mathcal{G}$.

\begin{proposition} [\cite{cf}, Lemma $4.10$] \label{cf}
The maximal components of $A$ are disjoint. There does not exist a path in $\mathcal{G}$ that begins in one maximal component and ends in another.
\end{proposition}

\begin{proof}
For the convenience of the reader, we include a sketch of the proof. Suppose there is a path of length $l$ between maximal components that starts at a vertex $x$ in $B_j$ and end at vertex $y$ in $B_k$. Then for large $n$, the number of  length $n$ paths that begin in $B_j$, traverse $m<n-l$ edges in $B_j$ to $x$, then follow our path to $y$ in $B_k$ and traverse $n-m-l $ edges in $B_k$, is growing like $\lambda^n$. Since we can vary $m$ between $1$ and $n-l$, this implies there are at least $C n \lambda^n$ paths from $B_j$ to $B_k$ for some $C>0$. This would imply $\#W_n$ grows at least like $ n\lambda^n$, contradicting Coornaert's estimates for $\#W_n$ \cite{coor}.
\end{proof}

This fact will be useful when counting certain quantities related to relative growth. To further facilitate these counting arguments, we define the following matrices.

\begin{definition}
For each $j=1,...,m$, define a matrix $C_j$ by,
\[
  C_j(u,v) = \left\{
     \begin{array}{@{}l@{\thinspace}l}
          0& \ \  \text{if $u$ or $v$ belong to a maximal component that is not $B_j$,}\\
          A(u,v) & \ \ \text{otherwise}.
     \end{array}
   \right.
\]
\end{definition}

Now suppose that $N \triangleleft G$ is a normal subgroup for which $G/N \cong \mathbb{Z}^\nu$ and let $\varphi: G \to G/N \cong \mathbb{Z}^\nu$ be the quotient homomorphism. We define a function $f: \Sigma_A \to \mathbb{Z}^\nu$ by 
$$f( (x_n)_{n=0}^\infty) = \varphi(\rho(x_0,x_1)),$$
where $\rho$ is the labeling map from Definition \ref{lab}. 
Since $f((x_n)_{n=0}^\infty)$ depends only on the first two coordinates of $(x_n)_{n=0}^\infty$, 
we can consider $f$ as a map from the directed edge set of $\mathcal{G}$ to $\mathbb{R}$. 
We then have that $\varphi(g) = f(\ast,x_1)+ f(x_1,x_2) + \cdots + f(x_{|g|-1},x_{|g|})$ where 
$(\ast,x_1),...,(x_{|g|-1},x_{|g|})$ is the unique path associated to $g$ by Property $(2)$ of Definition $2.1$. 
Using $f$, we weight the matrices $C_j$ componentwise and define, for $t \in \mathbb{R}^\nu$,
$$
C_j(t)(u,v) = e^{2\pi i \langle t, f(u,v) \rangle} C_j(u,v).
$$ 
We define the matrices $B_j(t)$ analogously.

\section{Proof of Theorem \ref{thm}}
Suppose $G$ is a non-elementary hyperbolic group and $N$ a normal subgroup satisfying the hypothesis of 
Theorem \ref{thm}. Let $\varphi: G \to \mathbb{Z}^\nu$ denote the quotient homomorphism. To study the relative growth of $N$, we would like to express 
$\#(W_n \cap N)$ in terms of the matrices $C_j(t)$. Using the orthogonality 
identity 
\[
\int_{\mathbb{R}^\nu/\mathbb{Z}^\nu} e^{2\pi i \langle t, \varphi(g) \rangle} \ dt = 
\begin{cases}
1 &\text{ if $\varphi(g)=0$}\\
0 &\text{ otherwise} 
\end{cases}
\]
we can write
\begin{align*}
\#(W_n \cap N) = \sum_{|g|=n} \int_{\mathbb{R}^\nu/\mathbb{Z}^\nu} 
e^{2\pi i \langle t, \varphi(g) \rangle} \ dt = \int_{\mathbb{R}^\nu/\mathbb{Z}^\nu} 
\sum_{|g|=n} e^{2\pi i \langle t, \varphi(g)\rangle} \ dt.
\end{align*}

The following result will allow us to rewrite $\#(W_n \cap N) $ in terms of the matrices $C_j$.
Let $v_\ast$ be the vector in $\mathbb{R}^V$ with a one in the coordinate corresponding to the 
$\ast$ vertex and zeros elsewhere. Also, let $\textbf{1} \in \mathbb{R}^\nu$ be the vector with a $1$ in each coordinate.
\begin{lemma}
There exists $\epsilon >0$ such that for all $t \in \mathbb{R}^\nu/\mathbb{Z}^\nu$
$$
\sum_{|g|=n} e^{2\pi i \langle t, \varphi(g) \rangle } 
= \sum_{j=1}^m  \langle C_j^n(t) v_\ast, \textup{\textbf{1}} \rangle + O((\lambda-\epsilon)^n)$$
as $n\to\infty$. The implied constant is independent of $t$.
\end{lemma}

\begin{proof}
Using the correspondence between $G$ and $\Sigma_A$, we can write
$$\left|\sum_{|g|=n} e^{2\pi i \langle t, \varphi(g) \rangle }  -\sum_{j=1}^m \langle C_j^n(t)v_\ast, \textup{\textbf{1}} \rangle \right| = (m-1) \left|\sum_{g \in M_n} e^{2 \pi i \langle t,\varphi(g) \rangle}\right| \le (m-1) \ \#M_n,$$
where $M_n$ consists of the elements in $G$ of word length $n$ whose corresponding path in $\mathcal{G}$ does not enter a maximal component. It is clear that $\#M_n = O((\lambda - \epsilon)^n)$ for some $\epsilon >0$ and so the result follows.
\end{proof}

Using this lemma, we see that
$$\#(W_n \cap N)  =\sum_{j=1}^m \int_{\mathbb{R}^\nu/\mathbb{Z}^\nu}   \langle C_j^n(t)v_\ast, \textup{\textbf{1}} \rangle \ dt 
+ O((\lambda-\epsilon)^n).$$
Hence to study the relative growth of $N$ would like to understand the spectral behaviour of the $C_j(t)$ for $t \in \mathbb{R}^\nu/\mathbb{Z}^\nu$. From their definitions, it is clear that the matrices $C_j$ each have $p_j$ simple maximal eigenvalues of modulus $\lambda$ and the rest of the spectrum is contained in a disk of radius strictly smaller than $\lambda - \epsilon$, for some $\epsilon >0$. We shall be interested in the values of $t$ for which the operators $C_j(t)$ have spectral radius $\lambda$. These values of $t$ are characterised by the following lemma.

\begin{lemma}
For any $t\in \mathbb{R}^\nu$, the operator $C_j(t)$ has spectral radius at most $\lambda$. Furthermore, $C_j(t)$ has spectral radius exactly $\lambda$ if and only if it has $p_i$ simple maximal eigenvalues of the form $e^{2\pi i \theta} e^{2\pi i k/p_i} \lambda$ for $k=0, \ldots ,p_i-1$ and some $\theta \in \mathbb{R}$. This occurs if and only if $B_j(t) = e^{2\pi i \theta} M B_j M^{-1}$ where $M$ is a diagonal matrix with modulus one diagonal entries. Furthermore, when $C_j(t)$ has $p_i$ simple maximal eigenvalues of modulus $\lambda$, the rest of the spectrum is contained in a disk of radius strictly less than $\lambda$.
\end{lemma}

\begin{proof}
When $C_j$ consists of a single component (ignoring the $\ast$ vertex) and so is the same as $B_j$, 
this is Wielandt's Theorem \cite{gant}. When this is not the case, we can write the spectrum of $C_j(t)$ 
as a union of the spectra of the irreducible components making up $C_j(t)$. By definition, each $C_j$ 
has one component $B_j$ with spectral radius $\lambda$ and all other components have spectral radius 
strictly less than $\lambda$. Therefore applying Wielandt's Theorem to each component gives the 
required result.
\end{proof}

We now follow the method presented in \cite{sharp}. Let $f_j = f|_{\Sigma_{B_i}}$ for $j=1,\ldots,m$. 
If a sequence $\gamma = (x_0,x_1,...,x_n)$ is such that $B_j(x_i,x_{i+1}) =1$ for $i=0,\ldots,n$ and 
$x_0 = x_n$, then we call $\gamma$ a cycle and define its length as $l(\gamma)=n$. Let $\mathcal{C}_j$ 
be the collection of all such cycles and note that the length of any cycle in $\mathcal{C}_j$ is a multiple 
of $p_j$. Given a cycle $\gamma \in \mathcal{C}_j$, we define its $f_j$-weight to be 
\[
w_{f_j}(\gamma) = f_j(x_0,x_1) + \cdots + f_j(x_{n-1},x_n).
\]
 Let $\Gamma_{j}$ be the subgroup of $\mathbb{Z}^\nu$ generated by 
 $\{ w_{f_j}(\gamma) : \gamma \in \mathcal{C}_j\}$. We define $\Delta_{j}$ 
 to be the following subgroup of $\Gamma_{f_j}$,
$$
\Delta_{j} = \{w_{f_j}(\gamma) - w_{f_j}(\gamma'): \gamma, \gamma' \in \mathcal{C}_j 
\text{ and } l(\gamma)=l(\gamma')\}.
$$
(This is a version of Krieger's $\Delta$-group \cite{krieger}. For a proof that it is a group,
see page 892 of \cite{sharp2001}.)
We now choose two cycles $\gamma, \gamma' \in \mathcal{C}_j$ such that $l(\gamma)-l(\gamma')=p_j$ 
and set $c_j = w_{f_j}(\gamma) - w_{f_j}(\gamma')$. 
Applying the results of \cite{mar-tun} to the aperiodic shift $(\Sigma_{B_j}, \sigma^{p_j})$, we see that the 
group $\Gamma_{j}/\Delta_{j}$ is cyclic and is generated by the element $c_j + \Delta_j$.
Our aim is to show that this group has finite order.  
To do so, we will use a result of Marcus and Tuncel.
For each $j=1,\dots,m$, let $E_j$ denote the directed edge set for the graph with transition matrix $B_j$. Write $V_j$ for the analogously defined vertex sets.
We say that a function $g: E_j \to \mathbb{R}$ is cohomologous to a constant if there exists $C \in \mathbb{R}$ and $h: V_j \to \mathbb{R}$ such that $g(x,y) = C + h(y)-h(x)$ for all $(x,y) \in E_j$.
\begin{lemma} [\cite{mar-tun}]
If $\langle t,f_j^{p_j}\rangle$ is not cohomologous to a constant for any non-zero $t \in \mathbb{R}^\nu/\mathbb{Z}^\nu$, then $\Gamma_j/\Delta_j$ has finite order.
\end{lemma}

It is clear that, for $t \in \mathbb{R}^\nu$, $\langle t,f_j^{p_j} \rangle$ is cohomologous to a constant if and only if $\langle t, f_j \rangle$ is cohomologous to  constant. Using ideas from \cite{stats}, we will show that the hypothesis of the above lemma is satisfied for each $j=1,\ldots,m$.

\begin{lemma} \label{notcohom}
For non-zero $t \in \mathbb{R}^\nu/\mathbb{Z}^\nu$ and for all $j=1,\ldots,m$, $\langle t,f_j\rangle$ is not cohomologous to a constant.
\end{lemma}

\begin{proof}
We begin by noting that, since $\varphi$ is surjective, for any $t\in\mathbb{R}^\nu\backslash\{0\}$ the function $\psi_t : =\langle t,\varphi \rangle : G \to \mathbb{R}$ is a non-trivial group homomorphism. Theorem $1.1$ and Proposition $7.2$ of \cite{stats} imply that if $\langle t, f_j \rangle$ (for any $j\in\{1,\ldots,m\}$) is cohomologous to a constant, then that constant is given by
$$\lim_{n\to\infty} \frac{1}{\#W_n} \sum_{|g|=n} \frac{\psi_t(g)}{n}.$$
Since our generating set $S$ is symmetric, $|g|=|g^{-1}|$ for all $g\in G$ and so the above limit is $0$ by symmetry. Hence we need to show that $\langle t, f_j \rangle$ is not cohomologous to $0$. By Livsic's criterion \cite{PP}, $\langle t, f_j \rangle$ is cohomologous to $0$ if and only if $\langle t, w_{f_j}(\gamma) \rangle = 0$ for all loops $\gamma \in C_j$.\\
\indent Suppose for contradiction that $\langle t, w_{f_j}(\gamma) \rangle = 0$ for all loops $\gamma \in C_j$. Now, for $\gamma=(x_0,\ldots,x_n) \in C_j$, $g_\gamma= \rho(x_0,x_1)\rho(x_1,x_2) \ldots \rho(x_{n-1},x_n)$ belongs to the kernel of $\psi_t$.  Furthermore, $g_\gamma$ has word length $n$. Also, Property $(2)$ from Definition $2.1$ implies that for any two distinct  loops $\gamma, \gamma' \in \mathcal{C}_j$, we have $g_\gamma \ne g_{\gamma'}$ whenever $\gamma$ and $\gamma'$ have the same initial vertex. Since the number of loops of length $np_j$ in $C_j$ is growing like $\lambda^{np_j}$, this implies that there exists $C>0$ such that
$$ \#(W_{np_j} \cap \text{ker}(\psi_t)) \ge C\lambda^{np_j}$$
for $n\ge1$ and hence that
$$\limsup_{n\to\infty} \frac{\#(W_n \cap \text{ker}(\psi_t))}{\#W_n} > 0.$$
Since $\text{ker}(\psi_t)$ is an infinite index subgroup of $G$, this contradicts the result of Gou\"ezel, Matheus and Maucourant \cite{gmm} written above as (\ref{gmm}).
\end{proof}

\begin{remark}
Since the above proof relies on the zero density result of Gou\"ezel, Matheus and Maucourant \cite{gmm}, quantifying the decay rate in (1.1) requires a priori knowledge of the convergence to zero.
\end{remark}

Let $D_j = \left|\Gamma_j/\Delta_j\right|$ for $j=1,\ldots,m$. From the above discussion, 
we know that each $D_j$ is finite. We also note that Lemma \ref{notcohom} shows that 
$\text{rank}_{\mathbb{Z}} ( \Gamma_j) = \nu$  and so $|\mathbb{Z}^\nu/\Gamma_j|$ is finite for each 
$j=1,\ldots,m$. Combining this with all of the above work, allows us to state
the following result that describes the spectral behaviour of the $C_j(t)$ as $t$ varies. 
We use the notation $\varrho(M)$ to denote the spectral radius of a matrix $M$.

\begin{proposition} \label{maxeig}
For $t \in \mathbb{R}^\nu/\mathbb{Z}^\nu$, define $\chi_t \in \widehat{\mathbb{Z}^\nu}$ by $\chi_t(x) = e^{2\pi i \langle t, x \rangle}$. Then we have that
$$\{ \chi_t: \varrho(C_j(t)) = \lambda\} = \Delta_{f_j}^\perp,$$ 
where $\Delta_{f_j}^\perp = \{ \chi \in \widehat{\mathbb{Z}^\nu} : \chi(\Delta_{f_j}) =1\}$. 
Furthermore, when $\chi_t \in \Delta_{f_j}^\perp$, $C_j(t)$ has $p_j$ simple maximal eigenvalues 
of the form $e^{2\pi i \theta} e^{2\pi i k/p_j}\lambda$ for some $\theta \in \mathbb{R}$ and 
$k=0,\ldots ,p_j-1$.
\end{proposition}

\begin{proof}
This is essentially Proposition $4$ from \cite{sharp} which is derived from work in \cite{ps1994}. 
However, here we need to consider the non-aperiodic matrices $C_j(t)$. 
To deduce this more general statement, we can apply Proposition $4$ from \cite{sharp} to the 
maximal component associated to the matrix $C_j^{p_j}(t)$. This is justified since this maximal 
component is aperiodic. To conclude the proof, we note that the part of the spectrum of $C_j(t)$ 
coming from $B_j(t)$ is invariant under the rotation $z \mapsto z e^{2\pi i/p_j}$.
\end{proof}

Proposition \ref{maxeig} implies that there exist $D_j < \infty$ values of $t$ for which the spectral 
radius of $C_j(t)$ is maximal and equal to $\lambda$. Denote these values by
$t=0,t_1^j\ldots ,t^j_{D_j-1}$. When $t$ takes one of these values, $C_j(t)$ has $p_j$ simple 
maximal eigenvalues of the form $e^{2\pi i \theta} e^{2\pi i k/p_j} \lambda$  for $k=0,\ldots ,p_j-1$ 
and for some 
$\theta \in \mathbb{R}$. 
We now choose, for each $j=1,\ldots ,m$, a neighbourhood $U_0^j$ of zero and define 
$U_r^j = U_0^j + t_k^j$ for $k=0,\ldots ,D_j-1$. 
Results from perturbation theory guarantee that, as long as each $U_0^j$ is sufficiently small, 
there exists $\epsilon >0$ such that the following hold for each $j=1,\ldots,m$.
\begin{enumerate}
 \item If $t \in \bigcup_{r=0}^{D_j-1}U_r^j$, then the matrices $C_j(t)$ each have $p_j$ simple, maximal  eigenvalues of the form $\lambda_j(t) e^{2\pi i k/p_j}$  for $k=0,\ldots,p_j-1$, where $t \to \lambda_j(t)$ is analytic and independent of $k=0,\ldots ,p_j-1$.
 \item  Let $M_\nu(\mathbb{C})$ denote the vector space of $\nu \times\nu$ complex matrices. For each $j=1,\dots,m$ and $k=0,\dots,p_j-1$, there exists an analytic matrix-valued function $Q_{j,k} : \bigcup_{r=0}^{D_j-1}U_r^j \to M_\nu(\mathbb{C})$,
 where $Q_{j,k}(t)$ is the eigenprojection onto the eigenspace 
 associated to the eigenvalue $\lambda_j(t)e^{2\pi i k/p_j}$ of the matrix $C_j(t)$.
 \item If $t \in (\mathbb{R}^\nu/\mathbb{Z}^\nu) \setminus  \bigcup_{r=0}^{Dj-1}U_r^j$
 then the spectral radius of each $C_j(t)$ is bounded uniformly above by $\lambda - \epsilon$.
\end{enumerate}

Using this description of the spectrum, we can write
$$
\#(W_n \cap N) = \sum_{j=1}^m \sum_{r=0}^{D_j -1} \sum_{k=0}^{p_j-1} \int_{U_r^j} \lambda_j(t) 
e^{2\pi i kn/p_j} \langle Q_{j,k}(t) v_\ast, \textbf{1} \rangle  \, dt + O((\lambda-\epsilon)^n),
$$ 
for some $\epsilon>0$. 
Hence there exists constants 
$c_{r,k}^j = \langle Q_{j,k}(t_r^j)v_\ast,\textbf{1}  \rangle$, for $r=0,\ldots ,D_j-1$ and $k=0,\ldots, p_j-1$, 
such that $\#(W_n \cap N)$ is equal to
\begin{equation}\label{bn}
\sum_{j=1}^m \left(\sum_{r=0}^{D_j-1} \sum_{k=0}^{p_j-1} e^{2\pi i n(r/D_j + k/p_j)} c_{r,k}^j  \right) \int_{U_0^j} \lambda_j(t)^n \left(1+O(\|t\|)\right) \, dt + O((\lambda-\epsilon)^n).
\end{equation}
The asymptotics of each
$$a_n^j := \int_{U_0^j} \lambda_j(t)^n \left(1+O(\|t\|)\right) \, dt$$
were studied in 
 \cite{ps1994}, where it was shown that, for each $j =1,\ldots,m$, there exists $\tau_j>0$ such that
\begin{equation} \label{asy}
 a_n^j \sim \frac{\tau_j \lambda^n}{n^{\nu/2}}
\end{equation}
as $n\to\infty$. 
Applying this along the subsequence $Dn$, where $D$ is given by the product of all the $p_1,\ldots, p_m$ and $D_1,\ldots, D_m$, we see that
\begin{equation} \label{bdn}
\#(W_{Dn} \cap N) = \frac{\widetilde{C}\lambda^{Dn}}{(Dn)^{\nu/2}} + o\left(\frac{\lambda^{Dn}}{(Dn)^{\nu/2}}\right)
\end{equation}
as $n\to\infty$, where
$$
\widetilde{C} =\sum_{j=1}^m \tau_j\left(\sum_{r=0}^{D_j-1} \sum_{k=0}^{p_j-1}c_{r,k}^j \right).
$$

It is clear that $\widetilde{C} \in \mathbb{R}_{\ge0}$. However, for $(\ref{bdn})$ to be a useful asymptotic expression, we would like that $\widetilde{C}$ is strictly positive. 
We now show that this is always the case.
\begin{lemma}
We necessarily have that $\widetilde{C}>0$.
\end{lemma}
\begin{proof}
Fix $j \in \{1,\ldots, m\}$ and recall that 
for any loop $\gamma =(x_0,\ldots,x_{Dn})  \in \mathcal{C}_j$ with $w_{f_j}(\gamma)=0$, the 
group element $g_\gamma = \rho(x_0,x_1)\rho(x_1,x_2) \ldots \rho(x_{Dn-1},x_{Dn})$ belongs 
to the kernel of $\varphi$ (or, equivalently, to $N$) and furthermore, $g_\gamma$ has word length $Dn$. Also, for any two distinct  loops 
$\gamma, \gamma' \in \mathcal{C}_j$, we have $g_\gamma \ne g_{\gamma'}$ whenever
$\gamma$ and $\gamma'$ have the same initial vertex. Combining these observations 
and applying the pigeonhole principle gives that
$$ \#(W_{Dn} \cap N ) \ge (\#V_j)^{-1} \#\{ \gamma \in \mathcal{C}_j : l(\gamma)=Dn, \, w_{f_j}(\gamma) =0 \}$$
for all $n \ge 1$. 
Pollicott and Sharp proved in \cite{ps1994} that
$$\#\{ \gamma \in \mathcal{C}_j : l(\gamma)=Dn, \, w_{f_j}(\gamma) =0 \} 
\sim \frac{K \lambda^{Dn}}{(Dn)^{\nu/2}}$$
as $n\to\infty$ for some $K >0$. Hence 
$$
\widetilde{C} = \limsup_{n\to\infty} \frac{(Dn)^{\nu/2} \#(W_{Dn} \cap N)}{ \lambda^{Dn}} \ge K(\#V_j)^{-1} >0,
$$
as required.
\end{proof}
We can now conclude the proof of our main result.
\begin{proof}[Proof of Theorem $1.1$]
Combining (\ref{bn}) and (\ref{asy}) implies that
$$\#(W_n \cap N) = O\left(\sum_{j=1}^m \int_{U_0^j} \lambda_j(t)^n \left(1+O(\|t\|)\right) \, dt \right) = O\left(\frac{\lambda^n}{n^{\nu/2}}\right)$$
which proves the first part of Theorem \ref{thm}. The second part follows from (\ref{bdn}) and the fact that $\widetilde{C} >0$.
\end{proof}


\begin{thebibliography}{11}

\bibitem{bourdon}
M. Bourdon,
Actions quasi-convexes d'un groupe hyperbolique, flot g\'eod\'esique, PhD Thesis, Universit\'e de
Paris-Sud, 1993.

\bibitem{bs}
R. Bowen and C. Series,
Markov maps associated with Fuchsian groups, 
Inst. Hautes \'Etudes Sci. Publ. Math. 50,
153--170, 1979.

\bibitem{cf}
D. Calegari and K. Fujiwara,
Combable functions, quasimorphisms, and the central limit theorem,
Ergodic Theory Dynam. Sys. 30, 1343--1369, 2010.


\bibitem{can} J. Cannon, The combinatorial structure of cocompact discrete hyperbolic groups, Geom. Ded. 
16, 123--148, 1984.



\bibitem{stats} S. Cantrell, Statistical limit laws for hyperbolic groups, arXiv:1905.08147 [math.DS], 2019.

\bibitem{coor}
M. Coornaert,
Mesures de Patterson--Sullivan sur le bord d'un espace hyperbolique au sens de Gromov,
Pacific J. Math. 159, 241--270, 1993.

\bibitem{cds}
R. Coulon, F. Dal'Bo and A. Sambusetti,
Growth gap in hyperbolic groups and amenability,
Geom. Funct. Anal. 28, 1260--1320, 2018.

\bibitem{fs}
P. Flajolet and R. Sedgewick,
Analytic Combinatorics,
Cambridge University Press, Cambridge, 2009.

\bibitem{gant} F. R. Gantmacher, The Theory of Matrices, Vol II, Chelsea, New York, 1974.

\bibitem{gh} \'E. Ghys and P. de la Harpe, Sur les groupes hyperboliques d'apr\`es Mikhael Gromov, Progress in Mathematics 83, Birkh\"auser, Boston, 1990.

\bibitem{gmm} S. Gou\"ezel, F. Math\`eus and F. Maucourant, Entropy and drift in word hyperbolic groups, Invent. math. 211, 1201--1255, 2018.

\bibitem{gri}
R. Grigorchuk,
Symmetrical random walks on discrete groups,
Multicomponent random systems (R. Dobrushin, Ya. Sinai and D. Griffeath, eds.), Advances 
in Probability and Related Topics 6, Dekker, 1980, 285--325.

\bibitem{gri-harpe}
R. Grigorchuk and P. de la Harpe,
On problems related to growth, entropy and spectrum in group theory,
J. Dynam. Control Systems 3, 51--89, 1997.

\bibitem{krieger}
W. Krieger,
On non-singular transformations of a measure space II,
Z. Wahrscheinlichkeitstheorie und Verw. Gebiete. 
11, 98--119, 1969.

\bibitem{mar-tun} B. Marcus and S. Tuncel, The weight-per-symbol polytope and scaffolds of invariants associated with Markov chains, Ergodic Theory Dynam. Sys. 11, 129--180, 1991.

\bibitem{PP} W. Parry and M. Pollicott, Zeta functions and periodic orbit structure of hyperbolic dynamics, Asterisque, 186--187, 1990.

\bibitem{ps1994} M. Pollicott and R. Sharp, Rates of recurrence for $\mathbb Z^q$ and $\mathbb R^q$ 
extensions of subshifts of finite type, J. London Math. Soc. 49, 401--416, 1994.

\bibitem{ps1996}
M. Pollicott and R. Sharp,
Growth series for the commutator subgroup,
Proc. Amer. Math. Soc. 124, 1329--1335, 1996.

\bibitem{sharp} R. Sharp, Relative growth series in some hyperbolic groups, Math. Ann. 312, 125--132, 1998.

\bibitem{sharp2001}
R. Sharp, 
Local limit theorems for free groups,
Math. Ann. 321, 889--904, 2001.


\end{thebibliography}
\end{document}